\documentclass[10pt]{article}

\usepackage[latin9]{inputenc}
\usepackage[T1]{fontenc}
\usepackage[french,english]{babel}
\usepackage{graphicx}
\usepackage[top=3cm, bottom=3cm, left=3cm, right=3cm]{geometry}
\usepackage{float}
\usepackage{amssymb}
\usepackage{amsmath}
\usepackage{titlepic}
\usepackage{caption} 
\usepackage{epstopdf}
\usepackage{xfrac}
\usepackage{subfigure}
\usepackage{bm}
\usepackage{esint}
\usepackage{mathtools}
\usepackage{upgreek}
\usepackage[dvipsnames]{xcolor}

\newcommand{\nm}{\noalign{\smallskip}}
\newcommand{\ds}{\displaystyle}

\usepackage{amsthm}

\newcommand{\R}{\mathbb{R}}
\newcommand{\N}{\mathbb{N}}

\newtheorem{remark}{Remark}[section]

\newtheorem{definition}{Definition}[section]
\newtheorem{lemma}{Lemma}[section]

\begin{document}

\title{Mathematical modelling of plasmonic strain sensors\thanks{\footnotesize
This work was supported in part by the Swiss National Science Foundation grant number
200021--172483.}}

\author{Habib Ammari\thanks{\footnotesize Department of Mathematics, 
ETH Z\"urich, 
R\"amistrasse 101, CH-8092 Z\"urich, Switzerland (habib.ammari@math.ethz.ch; alice.vanel@sam.math.ethz.ch).} \and Pierre Millien\thanks{\footnotesize  Institut Langevin, 1 Rue Jussieu, 75005 Paris, France (pierre.millien@espci.fr). }\and Alice L. Vanel\footnotemark[2] }

\date{}

\maketitle

\begin{abstract}
We provide a mathematical analysis for a metasurface constructed of plasmonic nanoparticles mounted periodically on the surface of a microcapsule. We derive an effective transmission condition, which exhibits resonances depending on the inter-particle distance. When the microcapsule is deformed, the resonances are shifted. 
We fully characterise the dependence of these resonances on the deformation of the microcapsule, enabling the detection of strains at the microscale level. 
We  present numerical simulations to validate our results.
\end{abstract}

\def\keywords2{\vspace{.5em}{\textbf{  Mathematics Subject Classification
(MSC2000).}~\,\relax}}
\def\endkeywords2{\par}
\keywords2{35R30, 35C20.}

\def\keywords{\vspace{.5em}{\textbf{ Keywords.}~\,\relax}}
\def\endkeywords{\par}
\keywords{plasmonic resonance, biomedical imaging, metasurface, strain sensing.}

\section{Introduction}
There is high value in early and real-time detection of deformation in materials. Most methods are invasive and do not allow for in-situ evaluation. In \cite{burel17}, gold nanoparticles were embedded on microcapsules subjected to uni-axial strain. Upon mechanical stress, the authors observed a change in color of the microcapsule. These experimental results motivate a mathematical modelling of the phenomenon. The aim of this paper is to use spectral analysis, building upon the works reported in \cite{abboud96} for non-resonating particles and \cite{metasurface_robin16} for a one-dimensional grating in the half-space, to derive a rigorous relation between the applied mechanical strain and the observed extinction shift.

Driven by the search of materials that achieve full control of wave propagation, the field of metamaterials has been undergoing considerable developments in the last decades. Metamaterials are artificial materials. Their building blocks are often locally resonant elements, whose features are an order of magnitude smaller than the operating wavelength. The properties, geometry, and size of the subwavelength resonant elements strongly alter the wave propagation in the structure. In this article, we show that the microcapsules synthesised in \cite{burel17} are in fact metasurfaces and that they owe their extraordinary sensing properties to the periodic arrangement of resonating gold nanoparticles embedded on their surface.

The desired optical effects are achieved by a phenomenon called surface plasmon resonance. Surface plasmon resonance occurs when the free electrons at the surface of a metal oscillate with a maximum amplitude. The surface plasmon resonance induces a strong absorption of the incident light by the nanoparticles. It is determined by a number of parameters: the nature of the metal, the dielectric properties of the background medium, the size, shape and configuration of the particles, among others; these allow a remarkably sophisticated degree of control over the desired optical response. The resonance is especially powerful and acute for noble metals, making gold a strong candidate for the sensitivity sensor. Another advantage of choosing gold is that its resonance occurs in the visible range of the electromagnetic spectrum [$380-740$nm], making the changes visible to the naked eye. In \cite{burel17}, the microcapsules have dimensions $4-30\upmu$m, and the nanoparticles have mean diameters $40$ nm $\pm 25$ nm; the resonators are indeed subwavelength. In \cite{ammari2017mathematicalscalar, ammari2016plasmaxwell,ando2016analysis}, the plasmonic resonances of a single particle are characterised in terms of the spectrum to some integral operator, known as the Neumann-Poincar\'e operator. 

In this paper, we mathematically formulate the scattering problem and derive an effective transmission condition using layer potential techniques (see Definition \ref{defNP}). We show that the absorption properties of the metasurface are associated with the eigenvalues of a periodic Neumann-Poincar\'e type operator. The thinness of the layer allows us to employ homogenisation techniques and effectively replace the plasmonic particles by an approximate transmission condition. In contrast with quasi-static plasmonic resonances of single nanoparticles, the quasi-static plasmonic resonances of periodically arranged nanoparticles depend on their size and configuration. We exploit well-known results on the periodic Green function and on the spectrum of its associated Neumann-Poincar\'e operator to fully characterise the resonances in terms of the structure's periodicity.

For simplicity, we reduce the problem to two dimensions where the microcapsule is a disk in the $\R^2$-plane and the nanoparticles are equally spaced disks mounted on its perimeter, see Figure \ref{fig:setting_capsule}. We assume that the nanoparticles stay equally spaced on the elongated ellipse. We model the propagation of light with the scalar wave equation and illuminate with a plane wave.

The rest of the paper is structured as follows: We begin in section \ref{sec:setting} by formulating the problem setting; in section \ref{sec:boundary_limit} we solve a hierarchy of equations from which we derive the effective transmission condition; next, we fully characterise the transmission condition in terms of the periodicity in section \ref{sec:strain}, and validate it against numerical simulations; we conclude in section \ref{sec:con} with a discussion of our results, generalization to three dimensions, and other future directions.

\begin{figure}
\centering
\includegraphics[trim={1.5cm 10.5cm 7cm 5.5cm},clip,width=15cm]{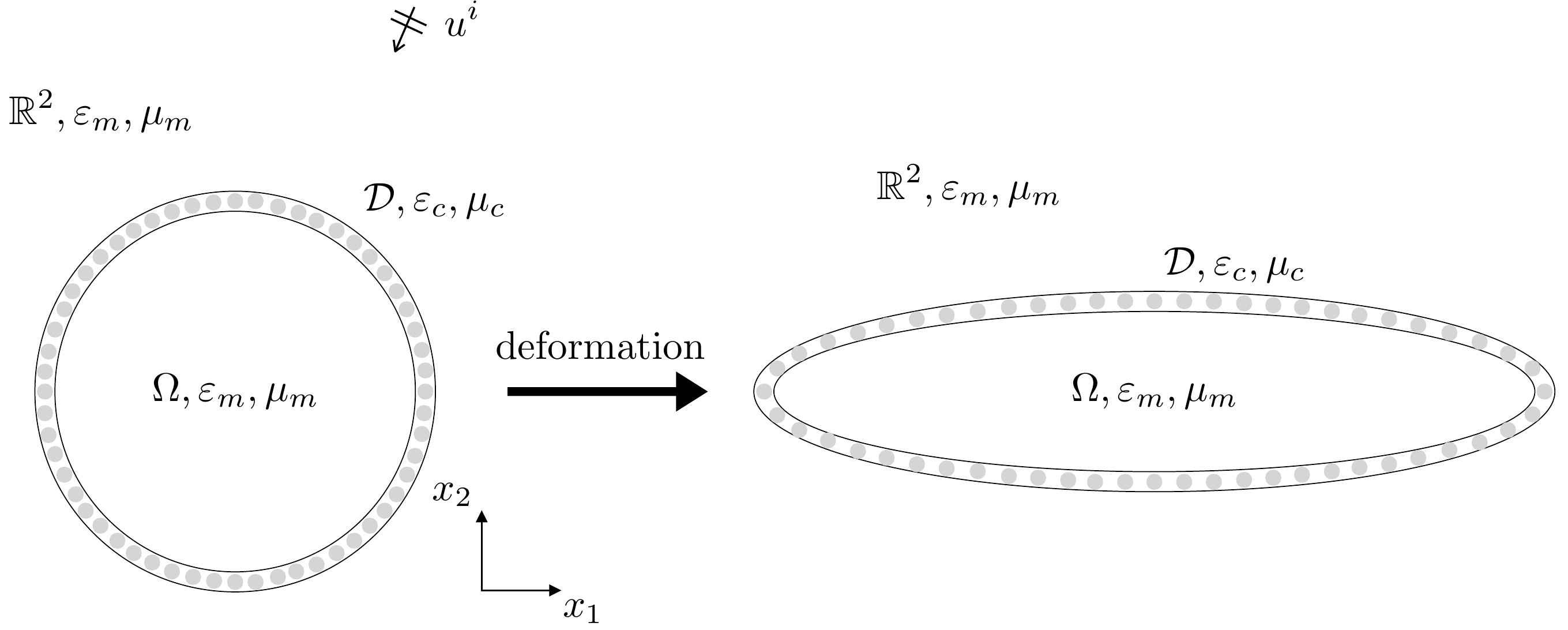}
  \caption{Not-to-scale schematic of the microcapsule before and after deformation.}
  \label{fig:setting_capsule}
\end{figure}

\section{Problem setting} \label{sec:setting}
We consider a particle $D$ occupying a smooth bounded domain in $\mathbb{R}^2$, of class $\mathcal{C}^{1,\alpha}$ for some $\alpha>0$, characterised by electric permittivity $\varepsilon_c$ and magnetic permeability $\mu_c$, both of which may depend on $\omega$, the frequency of the incoming wave. The particle $D$ has a characteristic size $\delta$ small compared to the operating wavelength. For ease of notation, we will write $\delta \ll 1$ in what follows instead of the correct homogeneous approximation $\delta \omega/c\ll1$, where $c$ is the speed of light in the medium. The background medium is characterised by its electric permittivity $\varepsilon_m$ and its magnetic permeability $\mu_m$. Throughout this paper, we assume that $\varepsilon_m$ and $\mu_m$ are real and positive and that $k_m$ is of order one. We also assume that $\Im{\varepsilon_c} \leq 0,\Re{\mu_c} \leq 0$ and $\Im{\mu_c} \leq 0$. We define the wavenumbers $k_c=\omega\sqrt{\varepsilon_c\mu_c}$ and $k_m=\omega\sqrt{\varepsilon_m\mu_m}$.

The particle is repeated periodically on the perimeter of a disk $\Omega$ of radius $r$ centred at the origin representing the microcapsule. Let $N$ be the number of particles and $\mathcal{D}=\cup_{l=1}^N D_l$ be the collection of periodically arranged particles. We assume the particles to be placed on the roots of unity, so their centres have coordinates in the complex plane $z_l=r\exp{(2i\pi l/N)}$, for $l=0,\ldots,N-1$, see Figure~\ref{fig:setting_capsule}. Let $\varepsilon=\varepsilon_c \chi(\mathcal{D})+\varepsilon_m \chi(\mathbb{R}^2 \setminus \bar{\mathcal{D}})$ and $\mu=\mu_c \chi(\mathcal{D})+\mu_m \chi(\mathbb{R}^2 \setminus \bar{\mathcal{D}})$, where $\chi$ denotes the characteristic function. Let $u^i(x)=\exp{(ik_m\kappa\cdot x)}$ be the incident wave, where $\kappa$ is the unit incidence direction.

Upon mechanical stress, the elastic circular microcapsule deforms into an ellipse. We choose the medium to be water, which is incompressible. For the microcapsule's surface to be conserved, its perimeter has to increase, which in turn increases, on average, the inter-particle distance. Not-to-scale circular and elliptic microcapsules are sketched in Figure~\ref{fig:setting_capsule}. In our approximation, the layer of particles stays periodic under the stretch and the period increases. Moreover, we consider the elastic layer on which the nanoparticles are mounted to be infinitely thin.

We use the Helmholtz equation to model the propagation of light. The total potential $u$ satisfies the following equations
\begin{equation}\label{eq:set-up}
\begin{dcases}
\nabla \cdot \left(\frac{1}{\mu} \nabla u\right) +\omega^2 \varepsilon u=0 & \mbox{in } \R^2, \\
\left. u\right|_+=\left.u\right|_- & \mbox{on } \partial \mathcal{D}, \\
\left. \frac{1}{\mu_m} \frac{\partial u}{\partial \nu}\right|_+ = \left. \frac{1}{\mu_c} \frac{\partial u}{\partial \nu}\right|_- & \mbox{on } \partial \mathcal{D},\\
\end{dcases}
\end{equation}
as well as the outgoing radiation condition
\begin{equation} \label{radiation}
\left| \frac{\partial(u-u^i)}{\partial |x|}-ik_m(u-u^i)\right|=\mathcal{O}\left(|x|^{-3/2}\right) \qquad  \mbox{as }  |x|\rightarrow \infty.
\end{equation}
Here, $\partial \cdot/\partial \nu$ denotes the normal derivative and the subscripts $+$ and $-$ are used to denote evaluation from outside and inside $\partial \mathcal{D}$, respectively.

It was shown in \cite{abboud96} that the curvature of the microcapsule does not appear at the first order approximation in $\delta$ of the periodic particles. So the approximate transmission conditions will be determined by studying the infinite periodic one-dimensional grating shown in Figure~\ref{fig:setting_line}(a), where now $\mathcal{D}=\cup_{l=1}^\infty D_l$. When considering the boundary layer, we use the scaled coordinates $\xi=(\xi_1,\xi_2)$,  as shown in panel (b) of the same figure, and denote by $d/\delta$ the inter-particle distance. The particles are repeated along the $\xi_1$-axis. We assume $d$ and $\delta$ to be of the same order of magnitude.

\begin{figure}
\centering
\includegraphics[trim={4cm 9cm 3.5cm 8.5cm},clip,width=15cm]{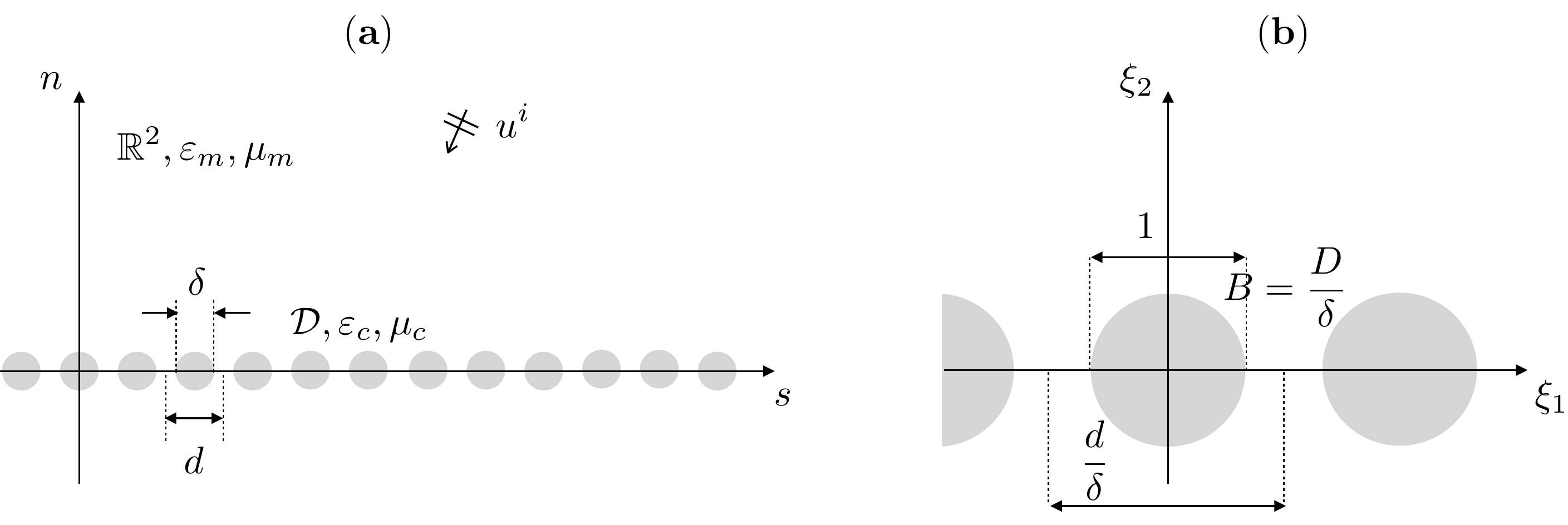}
 \caption{Plane grating of periodically arranged nanoparticles on (a), on (b) after stretched coordinates.}
\label{fig:setting_line}
\end{figure}

\section{Boundary layer approximation}
\label{sec:boundary_limit}
Assume that the capsule $\Omega$ is of class $\mathcal{C}^1$. For $x$ in a neighborhood of $\partial \Omega$, let $s$ be the curvilinear abscissa of the orthogonal projection of $x$ on $\partial \Omega$ and let $-n$ be the signed distance from $\Omega$. 
Following~\cite{abboud96}, for $x$ in a neighbourhood of $\partial \Omega$, we introduce the ansatz
\begin{equation} \label{eq:ansatz}
u(x)=u^{(0)}(x) +u^{(0)}_\text{BL}\left(\frac{s}{\delta}, \frac{n}{\delta}\right) +\delta\left(u^{(1)} (x) +u^{(1)}_\text{BL} \left(\frac{s}{\delta}, \frac{n}{\delta}\right) \right)+\mathcal{O}\left(\delta^2\right).
\end{equation}
Both $u^{(0)}$ and $u^{(1)}$ are solutions to a Helmholtz equation and do not satisfy boundary conditions on the plasmonic particles. The boundary-layer correctors $u^{(0)}_\text{BL}$ and $u^{(1)}_\text{BL}$ are introduced to correct the transmission conditions on the particles boundary, and are exponentially decaying as $|\xi_2|$  ($\xi_2 = n/\delta$) goes to infinity.
Note that the convergence of \eqref{eq:ansatz} was proved in \cite{abboud96} for the half-plane setting with Dirichlet boundary conditions.
The same arguments as those in \cite{abboud96} apply here.

By substituting the asymptotic expansion \eqref{eq:ansatz} into \eqref{eq:set-up} we find that the leading-order term $u^{(0)}$ solves 
\begin{equation}\label{eq:u^(0)}
\begin{dcases}
\Delta u^{(0)} +k_m^2 u^{(0)}=0 & \mbox{in } \R^2 \setminus \overline{\Omega} \mbox{ and } \Omega, \\
u^{(0)}-u^i  \mbox{ satisfies the outgoing radiation condition (\ref{radiation}) as } |x|\rightarrow \infty.
\end{dcases}
\end{equation}
The leading-order boundary-layer term $u^{(0)}_\text{BL}$ corrects the transmission conditions on $\partial \mathcal{D}$ up to order $\mathcal{O}(\delta)$ , and hence it solves
\begin{equation}\label{eq:u^(0)_BL}
\begin{dcases}
\nabla \cdot \left(\frac{1}{\mu} \nabla u^{(0)}_\text{BL} \right)+\omega^2\varepsilon u^{(0)}_\text{BL}=0 & \mbox{in } \left(\mathbb{R}^2\setminus\overline{\mathcal{D}}\right)\cup\mathcal{D}, \\
\left.  u^{(0)}_\text{BL}\right|_+=\left.  u^{(0)}_\text{BL} \right|_- & \mbox{on } \partial \mathcal{D}, \\
\left. \frac{1}{\mu_m} \frac{\partial  u^{(0)}_\text{BL}}{\partial \nu}\right|_+ - \left. \frac{1}{\mu_c} \frac{\partial  u^{(0)}_\text{BL}}{\partial \nu}\right|_-=\left(\frac{1}{\mu_c}-\frac{1}{\mu_m}\right)\frac{\partial  u^{(0)}}{\partial \nu} & \mbox{on } \partial \mathcal{D},\\
 u^{(0)}_\text{BL}  \text{ is exponentially decaying away from } \partial \Omega.
\end{dcases}
\end{equation}

We consider now a re-scaled problem where a particle occupies a bounded domain $B=D/\delta$ and is repeated periodically on the $\xi_1$-axis with period $d/\delta$, see Figure~\ref{fig:setting_line}(b). We denote by $\mathcal{B}$ the collection of these re-sized particles. We introduce two functions $\alpha^{(1)}$ and $\alpha^{(2)}$ and four complex constants $\alpha^{(1),+}_\infty,\alpha^{(1),-}_\infty,\alpha^{(2),+}_\infty,\alpha^{(2),-}_\infty$ that satisfy  in the variable $\xi=(\xi_1,\xi_2)$ and for $l=1,2$,
\begin{equation}\label{eq:alpha}
\begin{dcases}
\Delta \alpha^{(l)} =0 & \mbox{in } \left(\mathbb{R}^2\setminus\overline{\mathcal{B}}\right)\cup\mathcal{B}, \\
\left. \alpha^{(l)}\right|_+=\left.\alpha^{(l)}\right|_- & \mbox{on } \partial \mathcal{B}, \\
\left. \frac{1}{\mu_m} \frac{\partial \alpha^{(l)}}{\partial \nu}\right|_+ - \left. \frac{1}{\mu_c} \frac{\partial \alpha^{(l)}}{\partial \nu}\right|_-=\left(\frac{1}{\mu_c}-\frac{1}{\mu_m}\right)\nu_l & \mbox{on } \partial \mathcal{B},\\
\alpha^{(l)}-\alpha^{(l),+}_\infty \mbox{ is exponentially decaying as } \xi_2\rightarrow+\infty, \\
\alpha^{(l)}-\alpha^{(l),-}_\infty \mbox{ is exponentially decaying as } \xi_2\rightarrow-\infty.
\end{dcases}
\end{equation}
Then $u^{(0)}_\text{BL}$ defined by
\begin{equation*}\label{eq:u^(0)_BL_sol}
u^{(0)}_\text{BL}(x):=
\begin{dcases}
\delta\left[ \frac{\partial u^{(0)}}{\partial x_1}(x_1,0)\left( \alpha^{(1)}\left(\frac{x}{\delta}\right)- \alpha^{(1),+}_\infty \right)+ \frac{\partial u^{(0)}}{\partial x_2}(x_1,0)\left( \alpha^{(2)}\left(\frac{x}{\delta}\right)- \alpha^{(2),+}_\infty \right)\right] & \mbox{for } n \geq0, \\
\delta\left[ \frac{\partial u^{(0)}}{\partial x_1}(x_1,0)\left( \alpha^{(1)}\left(\frac{x}{\delta}\right)- \alpha^{(1),-}_\infty \right)+ \frac{\partial u^{(0)}}{\partial x_2}(x_1,0)\left( \alpha^{(2)}\left(\frac{x}{\delta}\right)- \alpha^{(2),-}_\infty \right)\right] & \mbox{for } n<0, \\
\end{dcases}
\end{equation*}
solves \eqref{eq:u^(0)_BL} up to order $\mathcal{O}(\delta^2)$.

\begin{lemma}
Using the periodic single-layer potential $\mathcal{S}_{B,\sharp }$ and the periodic Neumann-Poincar\'e operator $\mathcal{K}_{B,\sharp }^*$ defined in Appendix \ref{app:SLP}, we can write the solutions to \eqref{eq:alpha} as
\begin{equation}
\alpha^{(l)}=\mathcal{S}_{B,\sharp }\left(\lambda I-\mathcal{K}_{B,\sharp }^*\right)^{-1}\left[\nu_l\right] , \qquad l=1,2,
\end{equation}
where $I$ denotes the identity operator and the contrast $\lambda$ is given by
\begin{equation}
\lambda=\frac{\mu_m+\mu_c}{2(\mu_m-\mu_c)}.
\end{equation}
\end{lemma}

\begin{proof}
We search for densities $\Psi_l \in H^{-\frac{1}{2}}(\partial B)$ such that $\alpha^{(l)}=\mathcal{S}_{B,\sharp }[\Psi_l]$. From Lemma \ref{lem:symmetrization}, the periodic single-layer potential is harmonic in $\left(\mathbb{R}^2\setminus\overline{\mathcal{B}}\right)\cup\mathcal{B}$ and continuous across $\partial \mathcal{B}$ and so, we are left with the normal derivative jump condition.
\end{proof}
 
\begin{lemma}
The following expansions hold for $l=1,2$, $\xi=(\xi_1,\xi_2)$, 
\begin{align*}
& \alpha^{(l)}(\xi)=\alpha^{(l),+}_\infty+\mathcal{O}(\exp({-\xi_2})) &\mbox{as } \xi_2\rightarrow +\infty, \\
& \alpha^{(l)}(\xi)=\alpha^{(l),-}_\infty+\mathcal{O}(\exp({\xi_2})) &\mbox{as } \xi_2\rightarrow -\infty,\\
\end{align*}
with
\begin{align}
\alpha^{(1),+}_\infty&=-\alpha^{(1),-}_\infty=0,\\
\alpha^{(2),+}_\infty&=-\alpha^{(2),-}_\infty=-\frac{\delta}{2d}\sum_{j=1}^\infty \frac{\left\langle\phi_j,\nu_2\right\rangle_{\mathcal{H}_0^*(\partial B)}\left\langle\phi_j,\nu_2\right\rangle_{\mathcal{H}_0^*(\partial B)}}{(\lambda-\lambda_j)(\frac{1}{2}-\lambda_j)}, \label{eq:alpha2inf}
\end{align}
where $\{\lambda_j\}$ are the eigenvalues of $\mathcal{K}^*_{B,\sharp }$ and $\{\phi_j\}$ a corresponding orthonormal basis of eigenvectors.
\end{lemma}

\begin{proof} We use the expansions of the periodic Green's function derived in Lemma \ref{lem:G_hash}. As $\xi_2\rightarrow +\infty$, we have by definition
\begin{align*}
\alpha^{(l)}(\xi)&=\int_{\partial B} G_\sharp (\xi,\zeta)\left(\lambda I-\mathcal{K}_{B,\sharp }^*\right)^{-1}\left[\nu_l\right](\zeta)\mathrm{d}\sigma(\zeta),\\
&=\int_{\partial B} \left(\frac{\delta(\xi_2-\zeta_2)}{2d}-\frac{\ln{2}}{2\pi}\right)\left(\lambda I-\mathcal{K}_{B,\sharp }^*\right)^{-1}\left[\nu_l\right](\zeta)\mathrm{d}\sigma(\zeta)+\mathcal{O}\left(\exp({-\xi_2})\right),\\
&=-\frac{\delta}{2d}\int_{\partial B} \zeta_2\left(\lambda I-\mathcal{K}_{B,\sharp }^*\right)^{-1}\left[\nu_l\right](\zeta)\mathrm{d}\sigma(\zeta)+\mathcal{O}\left(\exp({-\xi_2})\right),\\
&=-\frac{\delta}{2d}\int_{\partial B} \zeta_2 \sum_{j=0}^\infty \frac{\left\langle\phi_j,\nu_l\right\rangle_{\mathcal{H}_0^*(\partial B)}}{\lambda-\lambda_j}\phi_j(\zeta)\mathrm{d}\sigma(\zeta)+\mathcal{O}\left(\exp({-\xi_2}) \right),\\
&=-\frac{\delta}{2d}\sum_{j=1}^\infty \frac{\left\langle\phi_j,\nu_l\right\rangle_{\mathcal{H}_0^*(\partial B)}\left\langle\phi_j,\zeta_2\right\rangle_{-1/2,1/2}}{\lambda-\lambda_j}+\mathcal{O}\left(\exp({-\xi_2}) \right),\\
&=-\frac{\delta}{2d}\sum_{j=1}^\infty \frac{\left\langle\phi_j,\nu_l\right\rangle_{\mathcal{H}_0^*(\partial B)}\left\langle\phi_j,\nu_2\right\rangle_{\mathcal{H}_0^*(\partial B)}}{(\lambda-\lambda_j)(\frac{1}{2}-\lambda_j)}+\mathcal{O}\left(\exp({-\xi_2})\right),
\end{align*}
where we used $\left\langle\phi_0,\nu_l\right\rangle_{\mathcal{H}_0^*(\partial B)}=0$ and
$$\int_{\partial B}\left(\lambda I-\mathcal{K}_{B,\sharp }^*\right)^{-1}\left[\nu_l\right](\zeta)\mathrm{d}\sigma(\zeta)=0.$$
Indeed, let $\Psi=\left(\lambda I-\mathcal{K}_{B,\sharp }^*\right)^{-1}\left[\nu_l\right]$. Since $\int_{\partial B}\nu_l \, \mathrm{d}\sigma=0$, we have $$\int_{\partial B}\left(\lambda I-\mathcal{K}_{B,\sharp }^*\right)\left[\Psi\right]\, \mathrm{d}\sigma=\int_{\partial B}\Psi\left(\lambda I-\mathcal{K}_{B,\sharp }\right)\left[1\right]\, \mathrm{d}\sigma =\left(\lambda-\frac{1}{2}\right)\int_{\partial B} \Psi =0\, \mathrm{d}\sigma.$$  Finally, $\lambda\neq1/2$, and so $\int_{\partial B} \Psi \, \mathrm{d}\sigma =0$. The last equality follows from
$$(\frac{1}{2}-\lambda_j)\left\langle\phi_j,\zeta_2\right\rangle_{-1/2,1/2}=\left\langle\phi_j,\nu_2\right\rangle_{\mathcal{H}_0^*(\partial B)}, $$ which is obtained by integration by parts.
We prove that $\alpha^{(1),+}_\infty=0$ by symmetry; indeed we have
\begin{align*}
\alpha^{(1),+}_\infty&=-\frac{\delta}{2d}\int_{\partial B} \zeta_2\left(\lambda I-\mathcal{K}_{B,\sharp }^*\right)^{-1}\left[\nu_1\right](\zeta)\mathrm{d}\sigma(\zeta),\\
&=-\frac{\delta}{2d}\left(\int_{\partial B^+} \zeta_2\left(\lambda I-\mathcal{K}_{B,\sharp }^*\right)^{-1}\left[\nu_1\right](\zeta)\mathrm{d}\sigma(\zeta)+\int_{\partial B^-} \zeta_2\left(\lambda I-\mathcal{K}_{B,\sharp }^*\right)^{-1}\left[\nu_1\right](\zeta)\mathrm{d}\sigma(\zeta)\right),
\end{align*}
where we split the boundary integral into an upper and lower half-space, on $\partial B^+$ and $\partial B^-$, respectively. A change of variable $\zeta'=-\zeta$ in the second integral gives $\zeta_2'=-\zeta_2$ and $\nu_1(\zeta')=\nu_1(\zeta)$ and hence,  the sum vanishes.

The proof for $\xi_2\rightarrow-\infty$ follows the same steps.
\end{proof}

Since there is no jump of $u$ across $\partial \Omega$,  $u^{(1)}$ must correct the jump of $u^{(0)}_\text{BL}$. Hence, the first-order term $u^{(1)}$ solves
\begin{equation}\label{eq:u^(1)}
\begin{dcases}
\Delta u^{(1)} +k_m^2 u^{(1)}=0 & \mbox{in } \R^2 \setminus \overline{\Omega} \mbox{ and } \Omega, \\
\nm
\left. \frac{\partial u^{(1)}}{\partial \nu} \right|_+ = \left. \frac{\partial u^{(1)}}{\partial \nu} \right|_- & \mbox{on } \partial \Omega, \\
\nm
\left.u^{(1)}\right|_{+}-\left.u^{(1)}\right|_{-}=-2\alpha^{(2),+}_\infty \frac{\partial u^{(0)}}{\partial \nu} & \mbox{on } \partial \Omega, \\
u^{(1)}-u^i \mbox{ satisfies the outgoing radiation condition (\ref{radiation}) at infinity.}
\end{dcases}
\end{equation}

\section{Effective transmission condition and strain sensing}
\label{sec:strain}
\subsection{Effective transmission condition}
By writing $u_\text{app}:=u^{(0)}+\delta u^{(1)}$ we find $u_\text{app}$ to be the solution of
\begin{equation}\label{eq:u_app}
\begin{dcases}
\Delta u_\text{app} +k_m^2 u_\text{app}=0 & \mbox{in } \R^2\setminus \overline{\Omega} \mbox{ and } \Omega, \\
\nm
\left.  \frac{\partial u_\text{app}}{\partial \nu} \right|_+ =  \left. \frac{\partial u_\text{app}}{\partial \nu} \right|_- & 
\mbox{on } \partial \Omega, \\
\nm
\left.u_\text{app}\right|_{+}-\left.u_\text{app}\right|_{-}=-2\delta\alpha^{(2),+}_\infty \frac{\partial u_\text{app}}{\partial \nu} & 
\mbox{on } \partial \Omega, \\
u_\text{app}-u^i \mbox{ satisfies the outgoing radiation condition (\ref{radiation}) at infinity.}
\end{dcases}
\end{equation}
We have derived an effective transmission condition on $\partial \Omega$, which is proportional to $\alpha^{(2),+}_\infty$. From  \eqref{eq:alpha2inf}, $\alpha^{(2),+}_\infty$ blows up at $\omega$ for which the spectrum of $\mathcal{K}^*_{B,\sharp}$ coincides with the contrast $\lambda(\omega)$. Notice that the eigenvalues $\{\lambda_j\}_{j\in\N}$ depend implicitly on the ratio $\delta/d$. So as the period increases, the frequency at which a plasmonic resonance occurs will be shifted to the right (i.e., the red). This follows from Lemma~\ref{lem:shape}. From equation \eqref{eq:shape}, it is clear that as the period increases, which is equivalent to the particle radius decreasing ($\eta<0$), eigenvalues $\{\lambda_j\}_{j\in\N}$ are larger.

The contrast can be written explicitly in terms of the frequency, using, for instance, the Drude model~\cite{ordal83}, to express the magnetic permeability of the particles as:
\begin{equation*}
\mu_c(\omega)=\mu_0\left(1-\frac{\omega_p^2}{\omega^2+i\omega\mathrm{T}^{-1}}\right),
\end{equation*}
where the positive constants $\omega_p$ and $\mathrm{T}$ are the plasma frequency and the collision frequency or damping factor, respectively. Here, $\mu_0$ is the magnetic permeability in vacuum. In the non-restrictive case where the medium is vacuum, i.e., $\mu_m=\mu_0$, the contrast has the simple expression
\begin{equation*}
\lambda(\omega)=\frac{\omega^2+i\omega\mathrm{T}^{-1}}{\omega_p^2}-\frac{1}{2}.
\end{equation*}
Now, solving $\lambda_j=\lambda(\omega)$ yields $(\Re \omega)^2=\left(\lambda_j+\frac{1}{2}\right)\omega_p^2-\mathrm{T}^{-2}/4$, which tells us that when the eigenvalues $\{\lambda_j\}_{j\in\N}$ are larger, the frequency is larger. As the wavelength is inversely proportional to the frequency, $\Lambda=2\pi c/\omega$, a period increase will shift the absorption peak to smaller wavelengths. This is consistent with the experimental results reported in~\cite{burel17,marzan2006}.

\subsection{Capsule's deformation}
The microcapsule's deformation under mechanical stress is characterised by the Taylor parameter $(D,\theta)$: a deformation index $D:=(L_1-L_2)/(L_1+L_2)$, where $L_1$ and $L_2$ are the major and minor axes of the ellipse and an orientation angle $\theta$ \cite{yuan2015,biesel2016,zheng2016}. In our particular case, the capsule's surface is conserved so $L_1=r^2/L_2$, where $r$ is the disk radius before elongation, and the strain is uni-axial on a film so $\theta=0$.
The perimeter of the ellipse can be approximated by $\mathcal{P}\approx\pi\sqrt{2}\sqrt{L_1^2+L_2^2}$. On the other hand, $\mathcal{P}\approx N d$, where $N$ was the number of nanoparticles. Therefore, by measuring the position of the absorption peak of the microcapsule, one can calculate the inter-particle distance and in turn fully characterise the deformation.

\subsection{Numerical illustration}
We now show numerical computations to further validate our results.
In \cite{burel17}, the capsules are roughly stretched by a factor of three, which corresponds to approximately doubling the inter-particle distance:
\begin{equation*}
\frac{\mathcal{P}'}{\mathcal{P}}=\frac{d'}{d}\approx \frac{\pi\sqrt{2}\sqrt{(3r)^2+(r/3)^2}}{2\pi r}\approx 2.13,
\end{equation*}
where $\mathcal{P}$ and $\mathcal{P}'$ are the perimeters of the circle and the ellipse, respectively.
\begin{figure}
\centering
  \includegraphics[trim={0.5cm 5.5cm 9cm 5.8cm},clip,width=12cm]{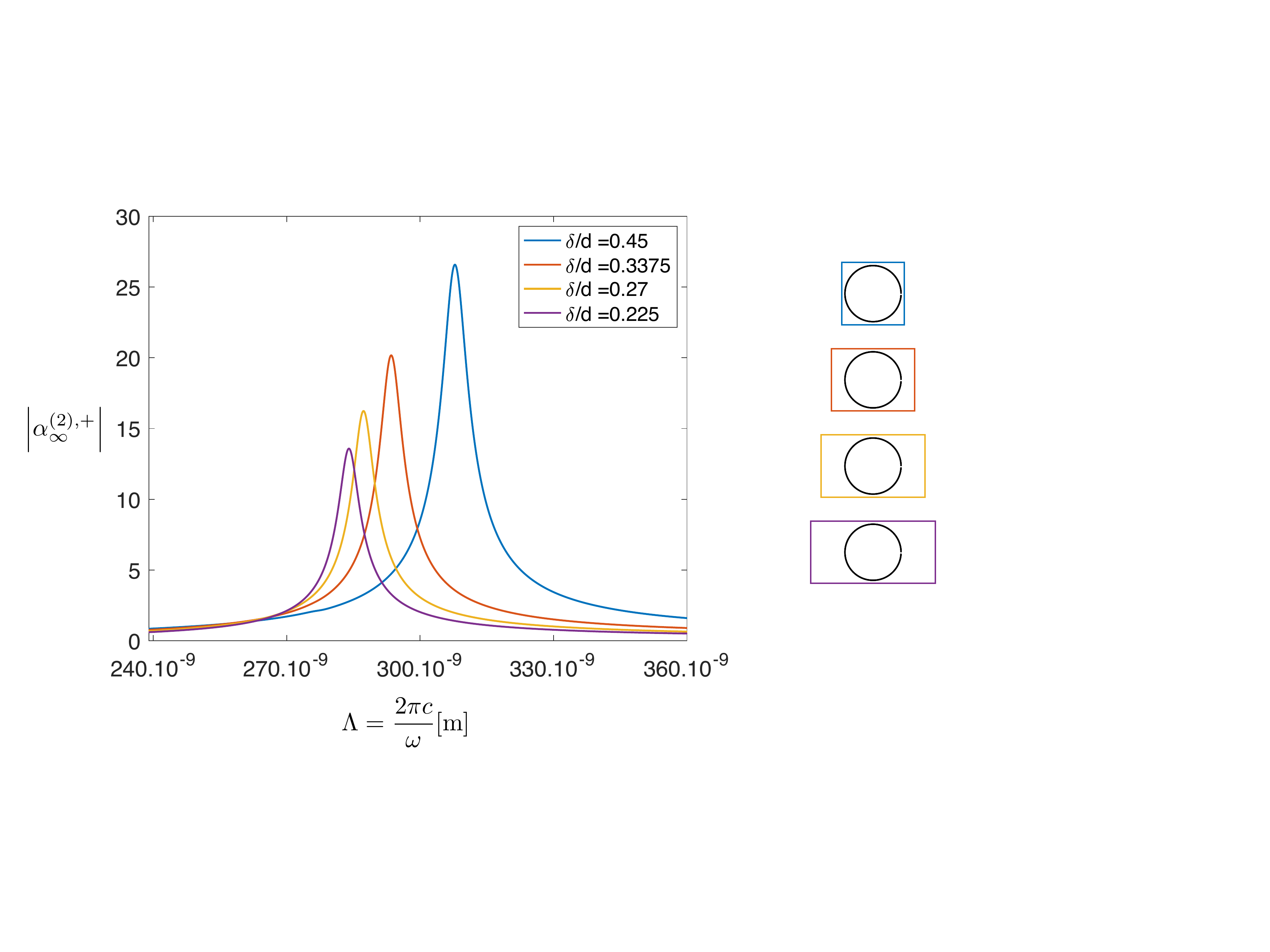}
  \caption{$\left|\alpha^{(2),+}_\infty\right|$ as a function of the wavelength for linearly increasing unit cell sizes, from $d=1$ (blue contour) to $d=2$ (purple contour), with a fixed radius $\delta=0.45$. Water was used for the homogeneous medium ($\varepsilon_m=(1.77)^2\varepsilon_0$) and gold for the nanoparticles. For the plasma frequency and damping factor we used $\mathrm{T}=10^{-14}$s and $\omega_p=2\cdot 10^{15}s^{-1}$. }
  \label{fig:diff_periods}
\end{figure}
Figure~\ref{fig:diff_periods} shows $\left|\alpha^{(2),+}_\infty\right|$ as a function of the wavelength for different periods but for a fixed radius. The larger the distance between the gold nanoparticles, the more red-shifted the plasmon peak is, which is consistent with our theoretical result in Lemma \ref{lem:shape}.

\begin{remark}
As the volume fraction of the nanoparticles increases, the absorption peak broadens and shifts to the red, as reported in \cite{marzan2006}, which explains why our absorption peaks are in the UV range and not in the visible range.
\end{remark}

\section{Conclusion}
\label{sec:con}
The mathematical modelling presented in this article gives a rigorous justification for the results reported in \cite{burel17}, where gold nanoparticles were used as building blocks to design strain sensing microcapsules. Using the spectral properties of the Neumann-Poincar\'e operator we derived an effective transmission condition and investigated the dependency of the effective transmission condition with respect to changes in the nanoparticles spacing.

Although the nanoparticles were modelled as disks, the calculations were conducted for an arbitrary shape, one with a sufficiently smooth boundary, and only the numerical computations shown in Figure~\ref{fig:diff_periods} are specific to circles.
This result could be extended to a two-dimensional array of spherical nanoparticles mounted on a two-dimensional surface.

\appendix

\section{Periodic Green's function}
\begin{definition} Let us define the one-dimensional periodic Green's function in $\mathbb{R}^2$ as the function $G_\sharp :\mathbb{R}^2\rightarrow \mathbb{C}$ satisfying
\begin{equation}\label{eq:D_G_hash}
\Delta G_\sharp (\xi)=\sum_{n\in \mathbb{Z}}\delta_0\left(\xi+\left(\frac{nd}{\delta},0\right)\right).
\end{equation}
\end{definition}

\begin{lemma}
Let $\xi=(\xi_1,\xi_2)$. Then 
\begin{equation}\label{eq:G_hash}
G_\sharp (\xi)=\frac{1}{4\pi}\ln{\left[\sinh^2\left(\frac{\pi\delta}{d} \xi_2\right)+\sin^2\left(\frac{\pi\delta}{d} \xi_1\right)\right]},
\end{equation}
satisfies \eqref{eq:D_G_hash}.
\end{lemma}

\begin{proof}
The proof can be found in \cite{photonic} in the special case $d/\delta=1$. Adding the multiplicative factor is straightforward.
\end{proof}
Let us denote by $G_\sharp (\xi,\zeta):=G_\sharp (\xi-\zeta)$.
\begin{lemma}\label{lem:G_hash}
The following expansions hold for $G_\sharp $ at infinity:
\begin{align*}
& G_\sharp (\xi)=\frac{\delta(\xi_2-\zeta_2)}{2d}-\frac{\ln{2}}{2\pi}+\mathcal{O}(\exp({-\xi_2})) &\mbox{as } \xi_2\rightarrow +\infty, \\
& G_\sharp (\xi)=\frac{-\delta(\xi_2-\zeta_2)}{2d}-\frac{\ln{2}}{2\pi}+\mathcal{O}(\exp({\xi_2})) &\mbox{as } \xi_2\rightarrow -\infty.
\end{align*}
\end{lemma}

\begin{proof}
As $\xi_2\rightarrow +\infty$, we have
\begin{align*}
G_\sharp (\xi,\zeta)&=\frac{1}{4\pi}\ln{\left[\sinh^2\left(\frac{\pi\delta}{d}(\xi_2-\zeta_2)\right)+\sin^2\left(\frac{\pi\delta}{d} (\xi_1-\zeta_1)\right)\right]},\\
&=\frac{1}{2\pi}\ln{\left[\sinh\left(\frac{\pi\delta}{d} |\xi_2-\zeta_2|\right)\right]}+\mathcal{O}\left(1+\frac{1}{\sinh^2(\xi_2)}\right),\\
&=\frac{1}{2\pi}\ln{\left[\exp\left(\frac{\pi\delta}{d}\left|\xi_2-\zeta_2\right|\right)-\exp\left(-\frac{\pi\delta}{d}\left|\xi_2-\zeta_2\right|\right)\right]}-\frac{\ln{2}}{2\pi}+\mathcal{O}\left(\ln\left(1+\exp({-2\xi_2})\right)\right),\\
&=\frac{1}{2\pi}\ln{\left[\exp\left(\frac{\pi\delta}{d}\left|\xi_2-\zeta_2\right|\right)\right]}-\frac{\ln{2}}{2\pi}+\mathcal{O}\left(\exp({-\xi_2})\right),\\
&=\frac{\delta(\xi_2-\zeta_2)}{2d}-\frac{\ln{2}}{2\pi}+\mathcal{O}(\exp({-\xi_2})).
\end{align*}
The proof is similar for $\xi_2\rightarrow -\infty$.
\end{proof}

\section{Periodic boundary integral operators}
\label{app:SLP}
In what follows, let $H^s(\partial B)$ be the usual Sobolev space of order $s$ on $\partial B$ and let $H_0$ denote the zero-mean subspace of $H$.
\begin{definition} \label{defNP}
We define the one-dimensional periodic single- and double-layer potentials and the one-dimensional periodic Neumann-Poincar\'e operator, respectively, for $B \Subset \left]-\frac{d}{2\delta},\frac{d}{2\delta}\right[ \times \mathbb{R}$ of class $\mathcal{C}^{1,\alpha}$ for some $0<\alpha<1$,
\begin{align*}
\mathcal{S}_{B,\sharp }:H^{-\frac{1}{2}}(\partial B) &\longrightarrow H_{\text{loc}}^1(\mathbb{R}^2),H^{\frac{1}{2}}(\partial B)\\
\phi &\longmapsto \mathcal{S}_{B,\sharp }[\phi](x)=\int_{\partial B} G_\sharp (x,y)\phi(y)\mathrm{d}\sigma(y), \quad x\in\mathbb{R}^2,~x\in \partial B; \\
\nm
\mathcal{D}_{B,\sharp }:H^{\frac{1}{2}}(\partial B) &\longrightarrow H_{\text{loc}}^1(\mathbb{R}^2), H^{\frac{1}{2}}(\partial B)\\
\phi &\longmapsto \mathcal{D}_{B,\sharp }[\phi](x)=\int_{\partial B} \frac{\partial G_\sharp (x,y)}{\partial \nu(y)} \phi(y)\mathrm{d}\sigma(y), \quad x\in\mathbb{R}^2 \setminus \partial B,~x \in \partial B; \\
\nm
\mathcal{K}^*_{B,\sharp }:H^{-\frac{1}{2}}(\partial B) &\longrightarrow H^{-\frac{1}{2}}(\partial B)\\
\phi &\longmapsto \mathcal{K}^*_{B,\sharp }[\phi](x)=\int_{\partial B} \frac{\partial G_\sharp (x,y)}{\partial \nu(x)}\phi(y)\mathrm{d}\sigma(y), \quad x\in \partial B.
\end{align*}
\end{definition}

\begin{lemma}\label{lem:symmetrization} We recall the following classical results \cite{photonic}.
\begin{enumerate}
\item[(i)] For any $\phi \in H^{-\frac{1}{2}}(\partial B)$, $\mathcal{S}_{B,\sharp }$ is harmonic in $B$ and in $\left]-\frac{d}{2\delta},\frac{d}{2\delta}\right[\times \R \setminus \overline{B}$.
\item[(ii)] The following Plemelj's symmetrization principle identity (also
known as Calder\'on's identity) holds:
\begin{align*}
\mathcal{K}_{B,\sharp } \mathcal{S}_{B,\sharp} = \mathcal{S}_{B,\sharp } \mathcal{K}_{B,\sharp }^* \qquad \mbox{on~} H^{-\frac{1}{2}}(\partial B),
\end{align*}
where $\mathcal{K}_{B,\sharp }$ is the $L^2$-adjoint of $\mathcal{K}_{B,\sharp }^*$.
\item[(iii)] The operator $\mathcal{K}_{B,\sharp }^*:H^{-\frac{1}{2}}_0(\partial B) \rightarrow H^{-\frac{1}{2}}_0(\partial B)$ is self-adjoint in the Hilbert space $\mathcal{H}_0^*(\partial B)$ which is $H^{-\frac{1}{2}}_0(\partial B)$ equipped with the following inner product:
\begin{equation*}
\left\langle u, v\right\rangle_{\mathcal{H}^*_0(\partial B)} = -\left\langle u, \mathcal{S}_{B,\sharp }[v]\right\rangle_{-\frac{1}{2},\frac{1}{2}},
\end{equation*}
with $-\left\langle \cdot, \cdot \right\rangle_{-\frac{1}{2},\frac{1}{2}}$ being the duality pairing between $H^{-\frac{1}{2}}_0(\partial B)$ and $H^{\frac{1}{2}}_0(\partial B)$, which makes $\mathcal{H}_0^*(\partial B)$ equivalent to $H^{-\frac{1}{2}}_0(\partial B)$.
\item[(iv)] If $\partial B$ is of class $\mathcal{C}^{1,\alpha}$, for some $\alpha>0$, then $\mathcal{K}_{B,\sharp }^*$ is compact. Let $(\lambda_j,\phi_j)_{j\in \N}$, be the eigenvalues and normalized eigenfunctions of $\mathcal{K}_{B,\sharp }^*$ in $\mathcal{H}^*(\partial B)$. Then $\lambda_j \in ]-1/2, 1/2]$, $\lambda_0=1/2$ and $\lambda_j \rightarrow 0$ as $j\rightarrow \infty$.
\item[(v)] Since $\mathcal{K}_{B,\sharp }[1]=1/2$, it holds that
\begin{equation*}
\int_{\partial B} \phi_j \, \mathrm{d}\sigma =0 \qquad  \mbox{for } j\neq 0.
\end{equation*}
\item[(vi)] The following trace formulae hold for $\phi \in H^{-\frac{1}{2}}(\partial B)$:
\begin{eqnarray*}
\left. \mathcal{S}_{B,\sharp }[\phi] \right|_+&=& \left. \mathcal{S}_{B,\sharp }[\phi] \right|_- ,\\
\nm
\left. \mathcal{D}_{B,\sharp }[\phi] \right|_{\pm}&=& \left(\mp\frac{1}{2}I+\mathcal{K}_{B,\sharp }\right)[\phi],\\
\nm
\left.\frac{\partial \mathcal{S}_{B,\sharp }[\phi]}{\partial \nu}\right|_{\pm}&=&\left(\pm\frac{1}{2}I+\mathcal{K}^{*}_{B,\sharp }\right)[\phi].
\end{eqnarray*}
\item[(vii)] The following representation formula holds:
\begin{align*}
\mathcal{K}_{B,\sharp }^*[\phi]=\sum_{l=0}^\infty \lambda_j \left\langle \phi,\phi_j\right\rangle_{\mathcal{H}^*_0(\partial {B})} \phi_j, \qquad \forall\phi \in \mathcal{H}^*_0(\partial B).
\end{align*}
\end{enumerate}
\end{lemma}

The following result on the shape derivative of the eigenvalues of $\mathcal{K}_{B,\sharp }^*$ follows from \cite{AKL}.
\begin{lemma} \label{lem:shape} Let $B_\eta =\{ x + \eta \nu(x), x\in \partial B\}$ for $|\eta|$ small enough. Suppose that $\lambda_j(B)$ is simple. Then 
\begin{equation} \label{eq:shape}
\lambda_j(B_\eta)= \lambda_j(B) - \eta\left(\lambda_j - \frac{1}{2}\right) \left(\lambda_j + \frac{1}{2}\right)\int_{\partial B}|\phi_j|^2 \, \mathrm{d}\sigma + \eta  \int_{\partial B}  \left| \frac{\partial \mathcal{S}_{D,\sharp }}{\partial T}[\phi_j] \right|^2 \, \mathrm{d}\sigma  + \mathcal{O}\left(\eta^2\right).
\end{equation}
\end{lemma}
\begin{proof} Following \cite[p.54]{AKL}, we have
$$
\mathcal{K}_{B_\eta,\sharp }^* = \mathcal{K}_{B,\sharp }^*+ \eta \left[ \frac{\partial \mathcal{D}_{B,\sharp }} {\partial \nu} - \frac{\partial^2 \mathcal{S}_{B,\sharp }}{\partial T^2} \right] + \mathcal{O}\left(\eta^2\right),
$$
where $\partial \cdot /\partial T$ denotes the tangential derivative. Therefore, since $\lambda_j$ is assumed to be simple, 
$$ \begin{array}{lll}
\lambda_j(B_\eta) &=& \ds \lambda_j(B) + \eta \left\langle \frac{\partial \mathcal{D}_{B,\sharp }[\phi_j]} {\partial \nu} - \frac{\partial^2 \mathcal{S}_{B,\sharp }[\phi_j]}{\partial T^2},\phi_j  \right\rangle_{\mathcal{H}^*_0(\partial B)} + \mathcal{O}\left(\eta^2\right)\\
\nm
&=& \ds \lambda_j(B) - \eta \int_{\partial B} \left[ \frac{\partial \mathcal{D}_{B,\sharp }[\phi_j]} {\partial \nu} - \frac{\partial^2 \mathcal{S}_{B,\sharp }[\phi_j]}{\partial T^2}  \right]  \mathcal{S}_{B,\sharp }[\phi_j]\, \mathrm{d}\sigma + \mathcal{O}\left(\eta^2\right)\\
\nm
&=&\ds \lambda_j(B) + \eta  \int_{\partial B}  \mathcal{D}_{B,\sharp }[\phi_j]  \frac{\partial \mathcal{S}_{B,\sharp }[\phi_j]}{\partial \nu } \, \mathrm{d}\sigma  -  \eta  \int_{\partial B}  \left| \frac{\partial \mathcal{S}_{B,\sharp }}{\partial T}[\phi_j] \right|^2 \, \mathrm{d}\sigma + \mathcal{O}\left(\eta^2\right),
\end{array}
$$
by a standard perturbation argument. Hence, using the jump relations in Lemma \ref{lem:symmetrization} (vi), it follows that 
$$
 \lambda_j(B_\eta) =  \ds \lambda_j(B) + \eta \left(\lambda_j - \frac{1}{2}\right) \left(\lambda_j + \frac{1}{2}\right) \int_{\partial B}|\phi_j|^2 \, \mathrm{d}\sigma - \eta  \int_{\partial B}  \left| \frac{\partial \mathcal{S}_{D,\sharp }}{\partial T}[\phi_j] \right|^2 \, \mathrm{d}\sigma  +\mathcal{O}\left(\eta^2\right).$$

\end{proof}

\bibliographystyle{abbrv}
\bibliography{/Users/avanel/Dropbox/references}

\end{document}